\documentclass[12pt]{article}
\usepackage[english]{babel}
\usepackage{amsmath, amssymb}
\usepackage{amsthm} 
\usepackage{enumerate} 
\usepackage{thm-restate}

\usepackage{url} 
\usepackage[T1]{fontenc}
\usepackage[dvipsnames]{xcolor}
\usepackage{multirow}
\usepackage{fullpage}

\usepackage{float}
\usepackage{graphicx}
\usepackage{nicefrac}
\usepackage{subcaption}

\usepackage{hyperref}
\usepackage{cleveref}

\newtheoremstyle{plainsl}%
	{\topsep}
	{\topsep}
	{\slshape} 
	{}
	{\normalfont\bfseries}
	{.}
	{ }
	{}

\theoremstyle{plainsl}
\newtheorem{Theorem}{Theorem}
\newtheorem{Lemma}[Theorem]{Lemma}
\newtheorem{Cor}[Theorem]{Corollary}

\newtheorem*{Example}{Example}
\newtheorem*{Remark}{Remark}

\def\sqr#1#2{{\vbox{\hrule height.#2pt
    \hbox{\vrule width.#2pt height#1pt \kern#1pt
        \vrule width.#2pt}\hrule height.#2pt}}}
\def\eqed{\sqr{7}{7}}

\def\qedd{%
    \ifmmode\eqno\eqed
    \else\nobreak\ \hfill\eqed\medbreak\fi}

\numberwithin{equation}{section}
\newcommand\comp[1]{{\mkern2mu\overline{\mkern-2mu#1}}}
\renewcommand{\epsilon}{\varepsilon}
\DeclareMathOperator\Tr{Tr}
\DeclareMathOperator{\diag}{diag}
\newcommand{\xbox}{\boxtimes}

\makeindex

\newcommand{\arxiv}[1]{\href{https://arxiv.org/abs/#1}{\texttt{arXiv:#1}}}

\title{Spectral approaches for \texorpdfstring{$d$-improper}{d-improper} chromatic number}
\author{Krystal Guo \and Ross J. Kang \and Gabri\"elle Zwaneveld}
\date{11 November 2024}
\begin{document}

\maketitle

\begingroup
    \renewcommand\thefootnote{}
     \footnotetext[0]{Korteweg-de Vries Institute for Mathematics, University of Amsterdam, Amsterdam, The Netherlands \texttt{\{k.guo,r.kang,g.c.zwaneveld\}@uva.nl}.  RJK was partially supported by the Gravitation Programme NETWORKS (024.002.003) of NWO.  }
\endgroup

\begin{abstract}

In this paper, we explore algebraic approaches to $d$-improper and $t$-clustered colourings, where the colouring constraints are relaxed to allow some monochromatic edges. 
Bilu [J.~Comb.~Theory Ser.~B, 96(4):608-613, 2006] proved a generalization of the Hoffman bound for $d$-improper colourings. We strengthen this theorem by characterizing the equality case. In particular, if the Hoffman bound is tight for a graph $G$, then the $d$-improper Hoffman bound is tight for the strong product $G \boxtimes K_{d+1}$. Moreover, we prove d-improper analogous for the inertia bound by Cvetkovíc and the multi-eigenvalue lower bounds of Elphick and Wocjan.

We conjecture an equality between the chromatic number of a graph $G$ and the \linebreak $d$-improper chromatic number of its strong product with a complete graph, $G \boxtimes K_{d+1}$, and prove the conjecture in special graph classes, including perfect graphs and graphs with chromatic number at most 4. Other supporting evidence for the conjecture includes a fractional analogue, a clustered analogue, and various spectral relaxations of the equality. 


\end{abstract}

\section{Introduction}
Improper and clustered colourings are both natural relaxations of proper colouring in that they allow monochromatic edges but only in some controlled fashion.
The first, called \textsl{improper (or defective) colouring}, permits a prescribed maximum degree in the monochromatic subgraphs. The second, called \textit{clustered colouring}, permits a prescribed maximum component size in the monochromatic subgraphs.

The idea of ``tolerating'' a controlled amount of error in proper colourings arises naturally in certain optimization settings. It occurs in the applied contexts of ad hoc radio communication networks~\cite{HKS09} and circuit design~\cite{Ait18phd}. Such notions have also been proven invaluable for the design of better distributed algorithms~\cite{BaEl13book}.
Within graph theory, such colourings have been studied going back to a paper of Lov\'asz in 1966~\cite{Lov66}. They have been considered in structural contexts~\cite{EKKOS15,HeWo18} and probabilistic/extremal contexts~\cite{ADOV03,KaMc10,GISW24}. See~\cite{Woo18survey} for a comprehensive recent survey.

The contribution of this paper is to leverage the use of algebraic methods here. 
We are partially motivated by recent interest in the corresponding notions for independent set of these colourings for various classes of symmetric graphs. This is a setting where the development of new algebraic methods may be helpful. A famous example is related to the so-called Sensitivity Conjecture: Huang's resolution of this conjecture showed an upper bound, in the Boolean hypercube, on the size of vertex subsets inducing some prescribed maximum degree. We refer to a preprint by Potechin and Tsang~\cite{PoTs24+} for an overview of some recent efforts to extend this result to other symmetric graphs. Another example is the treatment by Esperet and Wood~\cite{EsWo23+} of such colourings for certain strong products, motivated in part by product structure theory and several other interesting connections to other parts of mathematics, see~\cite[Subs.~1.2--1.7]{EsWo23+}.
These recent results compel a systematic exploration of algebraic methods for improper and clustered colourings.

Throughout the paper, we use $\chi^d(G)$ to denote the {\em $d$-improper chromatic number} of $G$, the least number of parts in a partition of the vertex set of $G$ so that each part induces a graph of maximum degree at most $d$. Note that $\chi^0(G)$ coincides with the usual chromatic number $\chi(G)$ of $G$.

Before discussing the contributions of this paper, we highlight a known generalization by Bilu~\cite{BILU2006608} of the well-known chromatic number bound of Hoffman~\cite{Hof70} (the classic result is the $d=0$ specialization below).
We write $\lambda_1 \geq \lambda_2 \geq \ldots \geq \lambda_n$ for the eigenvalues of the adjacency matrix of $G$. 
\begin{Theorem}[Bilu~\cite{BILU2006608}] \label{generalizedHoffmanbound}
  For every integer $d\geq 0$ and every graph $G$, we have \[\chi^d(G) \geq \frac{\lambda_1- \lambda_n}{d-\lambda_n} = 1- \frac{\lambda_1-d}{\lambda_n-d}.\]
\end{Theorem}
\noindent
In this vein of research, one of our contributions is to generalize other spectral bounds to $d$-improper colourings in Section \ref{Section: spectral bounds}. 

To grasp the Bilu bound better, we define a \textsl{$d$-improper Hoffman colouring} to be a vertex partition of $G$ into parts inducing maximum degree $d$ having $1- \frac{\lambda_1-d}{\lambda_n-d}$ parts, so that $G$ exactly attains the bound. For $d=0$, such colourings have already been long known to be rather special; for any fixed integer $m\ge 0$, there are only finitely many strongly regular graphs that admit a ($0$-improper) Hoffman colouring with $m$ colours~\cite{Hae79thesis}. Since the collection of subsets inducing subgraphs of maximum degree $d>0$ is richer than the collection of independent sets, it is intuitive to expect that a characterization of $d$-improper Hoffman colourings is more challenging for any $d>0$ than for the classic case $d=0$.

To this end, we consider the strong product\footnote{
The \textsl{strong product} $G \boxtimes H$ is the graph on vertex set $V(G)\times V(H)$ where $(u,u')(v,v')$ is an edge of $G \boxtimes H$ if (i) $u=v$ and $u'$ is adjacent to $v'$ in $H$, (ii) $u'=v'$ and $u$ is adjacent to $v$ in $G$, or (iii) $u$ is adjacent to $v$ in $G$ and $u'$ is adjacent to $v'$ in $H$.
} of $G$ with the complete graph $K_{d+1}$ on $d+1$ vertices, denoted $G \boxtimes K_{d+1}$. We show in Lemma \ref{Lem: construct Hoffman colourable} that $G \boxtimes K_{d+1}$ admits a $d$-improper Hoffman colouring when $G$ admits a proper Hoffman colouring. Incidentally, we also find that for such graphs $G$ we have $\chi(G) = \chi^d(G \boxtimes K_{d+1})$. In Section \ref{Section: Conjecture}, we dwell on the tantalizing prospect that moreover equality does hold for all $G$ as stated in the next conjecture:

\begin{restatable}{Con}{conjecturedimproper}\label{conjecture d improper} For every graph $G$ and every integer $d \geq 0$, we have
\begin{equation}\label{equation conjecture} \chi(G) = \chi^d(G \boxtimes K_{d+1}).\end{equation}
\end{restatable}
\noindent
The difficulty of this conjecture is to prove $\chi^d(G \boxtimes K_{d+1}) \geq \chi(G)$. To see the converse inequality $\chi(G) \ge \chi^d(G \boxtimes K_{d+1})$, consider any proper colouring $c : V(G) \to \mathbb{N}$ of $G$ and construct a $d$-improper colouring of $G \boxtimes K_{d+1}$ by using $c(v)$ on $(v,i)$, for each $i$. This is illustrated in Figure \ref{Fig: Lifted Colourings}. Since not every $d$-improper colouring of $G \boxtimes K_{d+1}$ comes from a proper colouring of $G$, it is far from obvious how a similar argument could be used for the converse. 

\begin{figure}
    \centering
    \begin{subfigure}[t]{0.3 \linewidth}
        \centering
        \raisebox{27 pt}{\includegraphics[width=0.6\linewidth,page= 2]{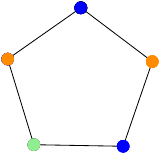}}
         \caption{Colouring of $C_5$}\label{C_5 colouring}
    \end{subfigure}\hfill
    \begin{subfigure}[t]{0.3 \linewidth}
        \centering
        \includegraphics[width=0.95\linewidth,page= 4]{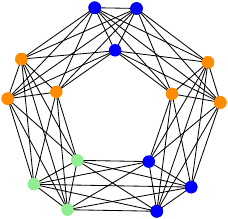}
         \caption{The lifted $2$-improper colouring of $C_5 \boxtimes K_3$ from the colouring of $C_5$ in (a).}
    \end{subfigure}\hfill
    \begin{subfigure}[t]{0.3 \linewidth}
        \centering
        \includegraphics[width=0.95\linewidth,page= 6]{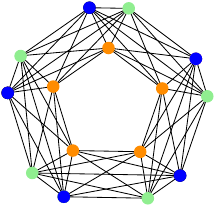}
         \caption{A $2$-improper colouring of $C_5 \boxtimes K_3$ that is not lifted from any colouring of $C_5$}
    \end{subfigure}
    \caption{Colourings of $G$ can be lifted to $d$-improper colourings of $G \boxtimes K_t$, but not all $d$-improper colourings of $G \boxtimes K_{d+1}$ are lifted from some colouring of $G$. In the black and white version, the colours are represented by different shades of gray.}\label{Fig: Lifted Colourings}
\end{figure}

A main part of our contribution is work supporting this conjecture. In particular, we have shown the conjecture to be true in the following special cases:
\begin{itemize}
\item $G$ has chromatic number at most $4$, which includes any planar $G$ (Lemma \ref{Lem: Pjotr max degree 4});
\item $G$ is a perfect graph; and 
\item $G$ attains equality in some of the  spectral bounds that we generalize in Section~\ref{Section: spectral bounds}.
\end{itemize}
Moreover, we have shown two weaker forms of Conjecture~\ref{conjecture d improper}:
\begin{itemize}
\item a fractional analogue: $\chi_f(G) = \chi_f^d(G \boxtimes K_{d+1})$, where $\chi_f^d$ is a natural fractional relaxation of the $d$-improper chromatic number (Lemma \ref{Cor frac makkelijk}); and 
\item a clustered analogue: $\chi(G) = \chi^{\underline{d+1}}(G \boxtimes K_{d+1})$, where $\chi^{\underline{t}}(G)$ denotes the least number of parts in a partition of the vertex set of $G$ so that each part induces components of size at most $t$ (Corollary \ref{cluster strong}).
\end{itemize}

We now discuss the organization of the paper. We give initial observations about $d$-improper colourings and strong products in Section \ref{Section: Preliminary results}. Section~\ref{Section: spectral bounds} 
 contains our generalizations of several well-known eigenvalue-based chromatic number lower bounds to the $d$-improper setting. 
In addition to an extended version of Bilu's Theorem that characterizes $d$-improper Hoffman colourings, we make detailed explorations for $d$-improper analogues of Cvetkovi\'{c}'s inertia bound~\cite{cvetkovic1972chromatic}, eigenvalue bounds discovered by Elphick and Wocjan \cite{wocjan2013new, elphick2015unified}.  
The development of Conjecture \ref{conjecture d improper} can be found in Section \ref{Section: Conjecture}. In particular, Section \ref{Section: Conjecture} contains the proof of the clustered version (Corollary \ref{cluster strong}).

In the remainder of this section, we will rigorously establish our definitions and preliminary remarks. 

\subsection{Definitions}\label{Section:Def}
A \textsl{(vertex-)colouring} of $G$ is a function $f: V(G) \rightarrow \mathbb{Z}^+$ where the image elements are referred to as \textsl{colours}. 
A colouring is called \textsl{proper} if any two adjacent vertices are assigned distinct colours.
The \textsl{chromatic number} of a graph $G$, denoted $\chi(G)$, is the least number of colours in a proper colouring of $G$. 

As mentioned in the introduction, our focus is the notion of a colouring that is \textsl{$d$-improper} (also \textsl{$d$-defective}). A colouring is $d$-improper if each vertex $v$ has at most $d$ neighbours of the same colour as $v$, or, equivalently, every monochromatic subgraph has maximum degree at most $d$. Notice that a $0$-improper colouring is a proper colouring. The \textsl{$d$-improper chromatic number}  $\chi^d(G)$ of a graph $G$ is the least number of colours in a $d$-improper colouring of $G$. It follows immediately that, if $s \geq t$, then
\[\chi^s(G) \leq \chi^t(G) \leq \chi(G)\] 
for all graphs $G$.

Another notion that we study is clustered colouring. A colouring is \textsl{$t$-clustered} if every monochromatic component contains at most $t$ vertices. The \textsl{$t$-clustered chromatic number} of a graph $G$, denoted $\chi^{\underline{t}}(G)$, is the least number of colours in a $t$-clustered colouring of $G$. A $1$-clustered colouring is a proper colouring. It follows immediately that, if $s \geq t$, then \[\chi^{\underline{s}}(G) \leq \chi^{\underline{t}}(G) \leq \chi(G)\] for all graphs $G$. Since every monochromatic component of $t$ vertices has maximum degree $t-1$, every $t$-clustered colouring is a $t-1$-improper colouring, and so, for all $G$, 
\[\chi^{\underline{t}}(G) \geq \chi^{t-1}(G).\]

To illustrate these concepts, Figure 2 contains examples of a proper, a $2$-improper and a $3$-clustered colouring of the Petersen graph.

\begin{figure}[ht]
    \centering
    \begin{subfigure}[t]{0.3 \linewidth}
        \centering
  {\includegraphics[width=0.95\linewidth,page= 1]{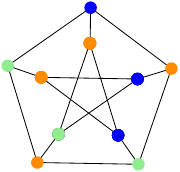}}
         \caption{A proper colouring of the Petersen graph.}\label{Petersen colouring}
    \end{subfigure}\hfill
    \begin{subfigure}[t]{0.3 \linewidth}
        \centering
        \includegraphics[width=0.95\linewidth,page= 2]{Petersengraph.pdf}
         \caption{A $2$-improper colouring of the Petersen graph.}
    \end{subfigure}\hfill
    \begin{subfigure}[t]{0.3 \linewidth}
        \centering
        \includegraphics[width=0.95\linewidth,page= 3]{Petersengraph.pdf}
         \caption{A $3$-clustered colouring of the Petersen graph.}
    \end{subfigure}
    \caption{Different colourings of the Petersen graph. In the black and white version, the colours are represented by different shades of gray.\label{Fig: colourings Petersen graph}}
\end{figure}

 In contrast to proper colourings, $d$-improper and $t$-clustered colourings are not easily described by the existence of graph homomorphisms from $G$ to a target graph.  Thus, some of the standard techniques for determining the chromatic number are not easily adaptable here.

A \textsl{clique} is a set of pairwise adjacent vertices and $\omega(G)$ denotes the maximum size of a clique in $G$. Since at most $d+1$ vertices in a clique can have the same colour in a $d$-improper colouring, we have that \begin{equation}\label{first inequality d-improper}
  \chi^{\underline{d+1}}(G) \geq \chi^d(G) \geq \frac{\omega(G)}{d+1}.
\end{equation}
Writing $\Delta(G)$ for the maximum degree of a graph, a notable upper bound on the $d$-improper chromatic number by Lov\'{a}sz \cite{Lov66} is the following:
\[\chi^d(G) \leq \left\lceil \frac{\Delta(G)+1}{d+1}\right \rceil.\]
We will give more bounds on the $d$-improper chromatic number in Section \ref{Section: spectral bounds}.

\medskip
As mentioned in the introduction, we are particularly interested in the $d$-improper and $(d+1)$-clustered chromatic numbers for the strong product $\chi(G \boxtimes K_{d+1})$ of the graph $G$ with $K_{d+1}$.
Recall that a \textsl{$b$-fold colouring} of a graph $G$ is a function $f: V(G) \rightarrow \binom{C}{b}$, i.e.~every vertex is assigned a subset of size $b$ of the colour set $C\subseteq \mathbb{Z}^+$. Such a colouring is proper if sets assigned to two adjacent vertices are disjoint. The \textsl{$b$-fold chromatic number} of a graph $G$, denoted $\chi_b(G)$, is the minimum number of colours in a proper $b$-fold colouring of $G$. Since $\chi(G \boxtimes K_b) = \chi_b(G)$, it is natural, when discussing such products, to also consider the $d$-improper and $t$-clustered analogues of the $b$-fold chromatic number.  

A $b$-fold colouring is \textsl{$d$-improper} if the graph induced by the vertices having any colour $x$ in their set has maximum degree at most $d$. Similarly, a $b$-fold colouring is \textsl{$t$-clustered} if the graph induced by the vertices having any colour $x$ in their set has component size at most $t$. The \textsl{$b$-fold $d$-improper chromatic number} and \textsl{$b$-fold $t$-clustered chromatic number} will be denoted by $\chi_b^d(G)$ and  $\chi_b^{\underline{t}}(G)$, respectively.

\begin{Remark}
In general, the equalities $\chi_b^d(G) = \chi^d(G \boxtimes K_b)$ and $\chi_b^{\underline{t}}(G)=\chi^d(G \boxtimes K_b)$ do not hold for all graphs $G$. For $v \in V(G)$, while the colours of $(v, i)$ may be the same for different $i$ in $G \boxtimes K_b$, the colours in the set assigned to $v$ should all be distinct. It follows, for example, that $\chi^1(C_4 \boxtimes K_2)=2$, but $\chi_2^1(C_4)>2$.
\end{Remark}

Recall that the \textsl{fractional chromatic number} of a graph $G$, denoted by $\chi_f(G)$, is defined by \[\chi_f(G) = \inf\left\{ \frac{\chi_b(G)}{b} : b \in \mathbb{N}\right\}.\] 
Analogously, the \textsl{$d$-improper fractional chromatic number} $\chi_f^d(G)$ and the \textsl{$t$-clustered fractional chromatic number} are defined by
\[\chi^d_{f}(G) := \inf\left\{ \frac{\chi^d_b(G)}{b} : b \in \mathbb{N}\right\} \quad \text{ and } \quad \chi^{\underline{t}}_{f}(G) := \inf\left\{ \frac{\chi^{\underline{t}}_b(G)}{b} : b \in \mathbb{N}\right\}.\]

\subsubsection*{Note on nomenclature}
In the  literature, authors use various terms and notation for  improper and clustered colourings.
For convenience, we have included a table with our chosen notation.

\begin{table}[ht]
 \centering
 \begin{tabular}{c|l}
  $\chi(G)$ &  chromatic number of a graph $G$ \\
  $\chi_b(G)$ &  $b$-fold chromatic number of a graph $G$\\
  $\chi_f(G)$ &  fractional chromatic number of a graph $G$ (only with an $f$)\\
  $\chi^d(G)$ &  $d$-improper chromatic number of a graph $G$\\
  $\chi^{\underline{t}}(G)$ &  $t$-clustered chromatic number of a graph $G$\\
 \end{tabular}
 \caption{Nomenclature of different chromatic numbers, as used in this paper.}
 \label{tab: notation}
\end{table}
\noindent
Other colouring parameters are analogously defined. For example, $\chi_f^d(G)$ denotes the fractional $d$-improper chromatic number of a graph $G$. 

\section{Preliminary observations}\label{Section: Preliminary results}

In this section, we will relate the $d$-improper chromatic number and the $d+1$-clustered chromatic number of the graph $G \boxtimes K_{d+1}$ to the chromatic number and the fractional chromatic number of $G$.
While many of the observations here are straightforward and/or known, they are useful and provide important context. We start by bounding $\chi^d(G \boxtimes K_{d+1})$ in terms of various chromatic numbers of $G$.

\begin{Lemma}[Lem.~14 in \cite{EsWo23+}]\label{colour ineq} For every graph $G$ and every integer $d \geq 0$, we have
 \[ \chi(G) \geq \chi^{\underline{d+1}}(G \boxtimes K_{d+1}) \geq \chi^d(G \boxtimes K_{d+1}) \geq \frac{\chi(G \boxtimes K_{d+1})}{d+1} = \frac{\chi_{d+1}(G)}{d+1} \geq \chi_f(G).\]
\end{Lemma}
\begin{proof}
 Every colouring of $G$ extends to a $(d+1)$-clustered colouring of $G \boxtimes K_{d+1}$ by giving every vertex $(v, i)$ the colour of $v$. Hence, $\chi(G) \geq \chi^{\underline{d+1}}(G \boxtimes K_{d+1})$. Secondly, $\chi^{\underline{d+1}}(G \boxtimes K_{d+1}) \geq \chi^{d}(G \boxtimes K_{d+1})$ since components of size $d+1$ have maximum degree $d$. For the third inequality, we observe that every colour class in a $d$-improper colouring can be properly coloured with at most $d+1$ colours. The last inequality follows from the definition of the fractional chromatic number.
\end{proof}

Using the same arguments, we can prove similar inequalities for the $b$-fold analogues of the chromatic numbers.

\begin{Lemma}\label{b fold ineq} For integers $b, d, k \geq 1$, we have
 \[ \frac{\chi_{b}(G)}{b} \geq \frac{\chi_{b}^{\underline{d+1}}(G \boxtimes K_{d+1})}{b} \geq \frac{\chi_{b}^d(G \boxtimes K_{d+1})}{b} \geq \frac{\chi_{ b}(G \boxtimes K_{d+1})}{b(d+1)} = \frac{\chi_{ b(d+1)}(G)}{b(d+1)} \geq \chi_f(G)\]
 for every graph $G$. \qedd
\end{Lemma}

As a consequence, we can now establish the fractional version of Conjecture \ref{conjecture d improper} in the following theorem.

\begin{Theorem}\label{Cor frac makkelijk} For every graph $G$ and all positive integers $d$, we have
 \[\chi_f^d(G \boxtimes K_{d+1}) = \chi_f(G).\]
\end{Theorem}
\begin{proof} Using \Cref{b fold ineq}, we obtain
\begin{align*}
 \chi_f(G) &= \inf \left\{ \frac{\chi_{b}(G)}{b} : b \in \mathbb{N} \right\}\\ & \geq \inf \left\{ \frac{\chi_b^d(G \boxtimes K_{d+1})}{b} : b \in \mathbb{N} \right\} = \chi_f^d(G \boxtimes K_{d+1})\\
 &\geq \inf \left\{ \chi_f(G) : b \in \mathbb{N} \right\} = \chi_f(G). \qedhere
\end{align*}
\end{proof}

By a completely analogous proof, we obtain that the fractional $t$-clustered chromatic number for $G \boxtimes K_{t}$ and the chromatic number of $G$ are equal.

\begin{Theorem}\label{fractional clustered} For every graph $G$ and all positive integers $t$, we have
 \[\chi_f(G) = \chi^{\underline{t}}_f(G\boxtimes K_t). \qedd\]
\end{Theorem}

We will show in  \Cref{Section: proof clustered} the non-fractional counterpart of \Cref{fractional clustered}. Next, we prove that for strong products $G \boxtimes K_{d+1}$, its $b$-fold $d$-improper chromatic number equals the $d$-improper chromatic number of its strong product with $K_b$.

\begin{Lemma}\label{b fold improper erin} For every integer $d\geq 0$ and graph $G$, we have
\[\chi^d_b(G \boxtimes K_{d+1}) = \chi^d(G \boxtimes K_{d+1} \boxtimes K_b).\]
\end{Lemma}
\begin{proof}
 Every $b$-fold $d$-improper colouring of $G \boxtimes K_{d+1}$ is also a $d$-improper colouring of $G \boxtimes K_{d+1} \boxtimes K_b$ by giving every vertex $(v, i)$ the $i$-th colour of $v$ (in some fixed ordering of the colour sets).

 Let $c$ be a $d$-improper colouring of $\chi^d(G \boxtimes K_{d+1} \boxtimes K_b)$. Since $(v, i_1 , i_2) \sim (v, j_1,j_2)$ for every $i_1 , i_2,j_1$ and $j_2$, every colour is at most in $d+1$ of these vertices in $G \boxtimes K_{d+1} \boxtimes K_b$. Thus, we can put the $b(d+1)$ colours of $v$ into $d+1$ sets of  $b$ different colours. 
\end{proof}

This holds analogously for the $b$-fold $t$-clustered chromatic number of $G \boxtimes K_t$.
\begin{Lemma}\label{b fold cluster erin}
 For every integer $t \geq 0$ and graph $G$, we have
\[\chi^{\underline{t}}_b(G \boxtimes K_{t}) = \chi^{\underline{t}}(G \boxtimes K_{t} \boxtimes K_b). \qedd \] 
\end{Lemma}

Lastly, we prove that the $t$-clustered chromatic number of graphs defined by the lexicographical product\footnote{
The \textsl{lexicographical product} of $G$ with $H$, denoted $G[H]$, has $G \times H$ as vertex set and $(v, i) \sim (w,j)$ if either $v \sim w$ or $v=w$ and $i \sim j$. In particular, $E(G \boxtimes H) \subseteq E(G[H])$ and there is an equality if and only if $H$ is a complete graph.
} $G[H]$ is only dependent on $G$ and the number of vertices of $H$.

\begin{Lemma}\label{lemma lexi prod} Let $H$ be a graph on $k$ vertices. Then for every graph $G$ and positive integer $t \geq k$, we have
\[\chi^{\underline{t}}(G[H]) = \chi^{\underline{t}}(G \boxtimes K_k) .\]
\end{Lemma}
\begin{proof}
Since $G[H]$ is a subgraph of $G \boxtimes K_k$, we have by monotonicity that \[\chi^{\underline{t}}(G[H]) \leq \chi^{\underline{t}}(G \boxtimes K_k).\]

 On the other hand, let $c$ be a $t$-clustered colouring of $G[H]$. If $v \neq w$ then both $G[H]$ and $G \boxtimes K_k$ have the property that $(v, i) \sim (w,j)$ if and only if $v \sim w$. Hence, if a component $C$ of colour $x$ in a $t$-clustered colouring of $G[H]$ contains vertices $(v, i)$ and $(w,j)$ such that $v \sim w$ in $G$, then all vertices $(v,k)$ with the colour $x$ also lie in $C$. Thus these monochromatic components do not increase in size if we add all missing edges $(v, i)(v, j)$.

 If there does not exist such $(v,i)$ and $(w,j)$ in $C$, then $C$ is equal to $\{(v, i) \mid i \in A \subseteq [n]\}$. This does not change when we add all missing edges between the different $(v, i)$. Hence, the component contains at most $k \leq t$ vertices in $G \boxtimes K_k$. In conclusion, every $t$-clustered colouring of $G[H]$ is also a $t$-clustered colouring of $G \boxtimes K_k$, whence the result follows. 
\end{proof}

\begin{Remark}
 Note that a similar statement to \Cref{lemma lexi prod} does not hold for the $d$-improper chromatic number. For example, take $d=2$, $G = K_6$ and $H= \comp{K_3}$. Then by the clique bound we have \[\chi^2(K_6 \boxtimes K_3)\geq \frac{\omega(K_6 \boxtimes K_3)}{3} = \frac{\omega(K_{18})}{3}=6.\] However, $\chi^{2}(K_6[\comp{K_3}]) \leq 5$ as it has the following $2$-improper colouring: 
 \begin{figure}[ht]
     \centering
     \includegraphics[width=0.25\linewidth,page=1]{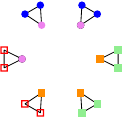}
     \caption{A $2$-improper colouring with 5 colours of $K_6[\comp{K_3}]$ drawn in its complement $\comp{K_6[\comp{K_3}]} = 6 K_3$. In the colouring of the complement graph, every vertex is non-adjacent to (at most) two vertices of its colour.}
 \end{figure}
 \end{Remark}

\section{Spectral bounds for the \texorpdfstring{$d$-improper}{d-improper} chromatic number}\label{Section: spectral bounds}

An important goal in graph theory is to develop good bounds on the chromatic number, which is NP-hard to compute, using graph invariants that are easier to compute, such as the degrees or graph eigenvalues. We have already seen a lower and an upper bound for the $d$-improper chromatic number, in \Cref{Section:Def}. In this section, we generalize some well-known spectral lower bounds for the chromatic number to the $d$-improper chromatic number. 

In the rest of the paper, we will sometimes need to be more specific by writing $\lambda_i(G)$ for the $i$th largest eigenvalue of $G$, and $A=A(G)$ for the adjacency matrix of $G$.
We will write $A$ for $A(G)$ and $\lambda_i$ for $\lambda_i(G)$ when the context is clear.

\subsection{\texorpdfstring{$d$-improper}{d-improper} Hoffman colourings}
One famous spectral bound for the chromatic number is the Hoffman bound \cite{Hof70}, also known as the ratio bound. This bound states that $\chi(G) \geq 1- \frac{\lambda_1}{\lambda_n}$. In this section, we will look at the generalization of this bound that holds for the $d$-improper chromatic number from \cite[Theorem 3]{BILU2006608}, which we will refer to as the Hoffman bound for $d$-improper colourings. Recall that, if the $d$-improper Hoffman bound is tight for some graph $G$, a $d$-improper colouring with $\chi^d(G)$ colours is called a $d$-improper Hoffman colouring of $G$. We will give a characterization of $d$-improper Hoffman colourings, which was not in \cite{BILU2006608}. 
In doing so, we give a stronger version of Bilu's theorem, with the equality characterization. To make notation a bit easier we use $m:=\chi^d(G)$.

 Given a partition $\mathcal{P}=\{V_1, \ldots, V_m\}$ of the vertices of a graph $G$ and a vector $x \in \mathbb{R}^{V(G)}$, we define the \textsl{weight-intersection numbers} for $u \in V_i$ with respect to $x$ as follows:
 \[b_{ij}(u) = \frac{1}{x_u} \sum_{v \in V_j \cap N(u)} x_v.\]

 A partition $\mathcal{P}$ is \textsl{weight-regular} whenever the weight-intersection numbers $b_{ij}(u)$ do not depend on the vertex $u \in V_i$, but only on the sets $V_i$ and $V_j$.

\begin{Theorem}\label{thm:Hoffmanncols}
     For every integer $d\geq 0$ and every graph $G$, we have \[\chi^d(G) \geq \frac{\lambda_1- \lambda_n}{d-\lambda_n} = 1- \frac{\lambda_1-d}{\lambda_n-d}.\]
     Furthermore, if equality holds, then the following hold:
     \begin{enumerate}[(a)]
         \item the multiplicity of $\lambda_n$ in $G$ is at least $m-1$;
         \item if $G$ is connected, then in any $d$-improper Hoffman colouring all  colour classes are $d$-regular and the partition into colour classes is weight-regular with respect to the weights $u_v$ where $u$ is the Perron eigenvector of $A$; 
         \item if $G$ has a unique Hoffman colouring up to permutation of the colour classes, then $\lambda_n$ has multiplicity  exactly $m-1$; and
         \item  if $G$ is regular, then the partition into $\chi^d(G)$ colour classes is equitable, and every vertex has $d$ neighbours in its colour class and $d-\lambda_n$ neighbours in every other colour class. 
     \end{enumerate}
\end{Theorem}

We will prove our main theorem in parts, for readability, but we will retain all definitions throughout the remainder of this section. First, we give a slightly adapted proof for the $d$-improper Hoffman bound of $\cite{BILU2006608}$, wherein we establish our notation for the rest of the section.

\begin{proof}[Proof of Bilu's Theorem]
 Let $X_1, \ldots, X_m$ be the colour classes of a $d$-improper colouring of $G$. By definition, the subgraphs induced by $X_i$ have maximum degree at most $d$. Moreover, let $u$ be the Perron eigenvector of $A$ of norm 1 and let $x_j$ be the entrywise product of $u$ with the characteristic vector of $X_j$.

We now define an $m\times m$ matrix $C$ where the $ij$-entry is given by
$$ C_{ij}= \frac{x_i^TAx_j}{||x_i||^2}.$$ By Cauchy interlacing,  the eigenvalues of $C$ interlace the eigenvalues of $A$. Moreover,
\begin{align*}
(C \mathbb{1})_i &= \frac{1}{||x_i||^2}\sum_{j=1}^m x_i^TAx_j = \frac{1}{||x_i||^2} x_i^TAu = \frac{\lambda_1}{||x_i||^2} x_i^Tu = \frac{\lambda_1}{||x_i||^2} x_i^Tx_i=\lambda_1.
\end{align*}
Thus, $\lambda_1$ is an eigenvalue of $C$ with eigenvector $\mathbb{1}$. 
Furthermore, the trace of $C$ equals
\[\Tr(C) = \sum_{i=1}^m \frac{1}{||x_i||^2} x_i^T A x_i= \frac{1}{||x_i||^2} x_i^T B_i x_i\] where $B_i$ is the block of $A$ indexed by the rows and columns of $X_i$. In other words, $B_i$ is the adjacency matrix of the subgraph induced by $X_i$. Since the maximum degree in $X_i$ is at most $d$, the largest eigenvalue of $B_i$ is at most $d$ and so 
\begin{equation}\label{eq:hoffmanbd-each-part-regular}
   x_i^TB_ix_i\leq d||x_i||^2. 
\end{equation}
Moreover, $\frac{1}{||x_i||^2} x_i^T B_i x_i$ is non-negative as $x_i$ only has non-negative entries. Therefore, we see $$0 \leq \Tr(C) \leq \sum_{i=1}^m d = dm,$$
where the second inequality holds with equality if and only if \eqref{eq:hoffmanbd-each-part-regular} holds with equality for all $i$. By interlacing we obtain
\begin{equation}\label{eq:hoffmanbd-sum-eigs}
    dm \geq \Tr(C) = \eta_1 + \ldots + \eta_m \geq \lambda_1+\lambda_{n-m+2}+\ldots \lambda_n \geq \lambda_1+(m-1)\lambda_n .
\end{equation}
The result follows from rewriting. 
\end{proof}

We will now characterize when there is equality in the Hoffman bound for $d$-improper colourings. As the eigenvalues and the chromatic number are determined by the different connected components, we will focus on the case when the graph is connected. 

\begin{proof}[Proof of parts (a) and (b) of \Cref{thm:Hoffmanncols}]
If equality holds in the Hoffman bound for $d$-improper colourings, then \eqref{eq:hoffmanbd-each-part-regular} holds with equality for every $i$ and \eqref{eq:hoffmanbd-sum-eigs} holds with equality. In particular, part (a) immediately follows since \eqref{eq:hoffmanbd-sum-eigs} gives that $\lambda_n$ has multiplicity at least $m-1$.

For each $i$, equality in \eqref{eq:hoffmanbd-each-part-regular} gives us that $d$ is an eigenvalue of $B_i$ with eigenvector $x_i$, which is a positive vector (on $X_i$) by Perron--Frobenius. Since the maximum degree of the subgraph induced by $X_i$ is at most $d$, we see that $d$ is the largest eigenvalue and every component of the subgraph induced by $X_i$ has degree $d$.  

 It remains to show that the partition into colour classes is weight-regular with respect to the weights $u_v$. 
 Recall that $\mathbb{1}$ is an eigenvector of $C$ with eigenvalue $\lambda_1 = dm-(m-1)\lambda_n$ and all other eigenvalues of $C$ are equal to $\lambda_n$. Since $C$ is diagonalizable, the eigenspace corresponding to $\lambda_n$ consists of all vectors orthogonal to $\mathbf{1}$. Thus the spectral idempotents of $C$ are $J_{m,m}$, the all ones matrix of order $m\times m$, and $I_m - J_{m,m}$. 
Thus $C$ equals
 \[C = \begin{pmatrix}
 d & d-\lambda_n& d-\lambda_n &\ldots & d-\lambda_n\\
 d-\lambda_n & d& d-\lambda_n &\ldots & d-\lambda_n\\
 \vdots & & \ddots & & \vdots\\
  \vdots & & & \ddots & \vdots\\
 d-\lambda_n & d-\lambda_n& d-\lambda_n & \ldots & d
\end{pmatrix}.\]
In particular, we have that \[\frac{x_i^TAx_j}{||x_i||^2} = C_{ij}= d-\lambda_n = C_{ji} = \frac{x_i^TAx_j}{||x_j||^2}.\] As $d-\lambda_n > 0$, we obtain that $||x_i||=||x_j||$ for all $i,j$. Thus, $$ \frac{x_i^T A x_j}{||x_i||^2} = \frac{x_i^T A x_j}{||x_i|| ||x_j||} = Q_{i,j},$$ where $Q$ is the normalized weight quotient matrix with respect to the weights $u_v$ where $u$ is the normalized Perron eigenvector of $A$. If a graph $G$ is $d$-improper Hoffman colourable, then the interlacing between $Q=C$ and $A$ is tight. Now it follows from  \cite[Lemma 2.3]{fiol1999eigenvalue} that this partition is weight-regular with respect to the weights $u_v$. \end{proof}

\begin{proof}[Proof of part (c) of \Cref{thm:Hoffmanncols}]
The tightness in the interlacing bound implies that $$Ax_i = \sum_{k=1}^m C_{ik}x_k$$ for all $i$ by  \cite[Corollary 2.5.3]{brouwer2011spectra}. 
It follows that, \[A(x_i-x_j) = \sum_{k=1}^m (C_{ik} x_i-C_{jk}x_j) = \lambda_n(x_i - x_j).\]
Thus $x_i-x_j$ is an eigenvector of $A$ with eigenvalue $\lambda_n$ for all $i,j$. These vectors span a space of dimension $m-1$. If there is some other colouring $c'$ of $G$, there are also eigenvectors $y_i-y_j$ of $A$ corresponding to $\lambda$ corresponding to $c'$. If $c'$ is not a permutation of the colours of our original colouring, at least one of these vectors $y_i-y_j$ will not lie in the span of $x_i-x_j$'s. Thus, the eigenspace corresponding to $\lambda_n$ contains a subspace of dimension $m$ so the multiplicity of $\lambda_n$ at least $m$.
\end{proof}

If the graph $G$ is regular, we can say even more about the equality case and prove part (d) by generalizing \cite[Lemma 2.3]{blokhuis20073} to $d$-improper Hoffman colourable graphs. 

\begin{proof}[Proof of part (d) of \Cref{thm:Hoffmanncols}] Recall that, in this case, every colour class $X_i$ is $d$-regular. If $G$ is regular, the Perron eigenvector $u = \mathbf{1}$ and thus $x_i$ is equal to the characteristic vector $\mathbf{1}_{X_i}$. Since the interlacing is tight, the partition into colour classes $X_1, \ldots, X_m$ is equitable, which is to say that, for every pair $i,j$, every element in $X_i$ is connected to exactly $C_{i,j}$ elements in $X_j$. The number of edges between $X_i$ and $X_j$ is nonzero and equal to both $|X_i|(d-\lambda_n) $ and $|X_j|(d-\lambda_n)$ and we obtain $|X_i|=|X_j|$ for all $i,j$. 
\end{proof}

The lemma below gives us a way to construct $d$-improper Hoffman colourable graphs from Hoffman colourable graphs, and this already generates a rich family of $d$-improper Hoffman colourable graphs.

\begin{Lemma}\label{Lem: construct Hoffman colourable}
 If $G$ admits a Hoffman colouring, then $G \boxtimes K_{d+1}$ admits a $d$-improper Hoffman colouring.
\end{Lemma}
\begin{proof}
 If $G$ admits a Hoffman colouring, then $\chi(G) = 1 - \frac{\lambda_1(G)}{\lambda_n(G)}$. Moreover, every eigenvalue of $G \boxtimes K_{d+1}$ is either $-1$ or $(d+1)\lambda_i(G)+d$. Thus, since $-1$ is not largest nor smallest eigenvalue of $G \boxtimes K_{d+1}$, we can see that 
\[ \lambda_1(G \boxtimes K_{d+1}) = (d+1)\lambda_1(G)+d , \quad \lambda_n(G \boxtimes K_{d+1}) = (d+1)\lambda_n(G)+d .\]  Therefore,
 \begin{align*}
     1- \frac{\lambda_1(G)}{\lambda_n(G)} &= \chi(G) \geq \chi^d(G \boxtimes H) \geq \frac{\lambda_1(G \boxtimes K_{d+1}) -\lambda_n(G \boxtimes K_{d+1})}{d-\lambda_n(G \boxtimes K_{d+1})}\\ &=\frac{(d+1)(\lambda_1(G) - \lambda_n(G))}{d-(d+1)\lambda_n(G)-d} = 1- \frac{\lambda_1(G)}{\lambda_n(G)}.\end{align*} So every inequality above must hold with equality and the result follows.
\end{proof}

\begin{Remark}
Suppose we take the product of a Hoffman colourable graph $G$ with another $d$-regular graph $H$. Again $\chi^d(G\boxtimes H) \leq \chi(G)$ since we can colour every vertex $(v,w)$ with the colour $c(v)$ for a proper colouring $c$ of $G$.

The largest eigenvalue of $G \boxtimes H$ is $(d+1)\lambda_1(G)+d$ and the smallest eigenvalue is \[\min\{(d+1)\lambda_n(G)+d, (\lambda_n(H)+1) \lambda_1(G)+\lambda_n(H)\}.\] If $(d+1)\lambda_n+d$ is the minimum quantity, then we again obtain a $d$-improper Hoffman colourable graph. If, however, the other quantity is strictly smaller, then there is no equality as $\frac{a+x}{b+x}$ is a decreasing function in $x > 0$ for $a>b$.
\end{Remark}

\begin{Example}
    Here we give an example of a Hoffman colourable $G$ and a $d$-regular graph $H$ such that $G \boxtimes H$ has no $d$-improper Hoffman colouring. Take $G=C_4$ (since $G$ is bipartite, it has a Hoffman colouring) and $H=K_{d,d}$. Then \[\lambda_1(C_4 \boxtimes K_{d,d})=3d+2, \lambda_n(C_4 \boxtimes K_{d,d})= \min\{-2(d+1)+d, 2(-d+1)-d\}= -3d+2.\]
   Moreover, $\chi^d(C_4 \boxtimes K_{d,d})=2$ since it has maximum degree $3d+2$ implying we need at least two colours and we can extend the colouring of $C_4$ to $C_4 \boxtimes K_{d,d}$. Thus 
    \[2  = \chi^d(C_4 \boxtimes K_{d,d}) > \frac{6d}{4d+2}\] 
    implying that $C_4 \boxtimes K_{d,d}$ has no Hoffman colouring.
\end{Example}

Below we construct an infinite family of $d$-improper Hoffman colourable graphs, without using strong products.

\begin{Example}
Let $H$ be a $2k$-regular graph with $2$-factorization $C_1,\ldots, C_k$; here the $C_i$s are edge-disjoint, $2$-regular, spanning subgraphs.  Let $G= L(H)$. We see that $G$ is $2(2k-1)$-regular and thus $\lambda_1(G) = 2(2k-1)$ and $\lambda_n = -2$, since $G$ is a line graph. Moreover, $\chi^2(G) \leq k$ as we can take $E(C_i)$ as our colour classes. Then we see that 
\[\chi^2(G) \leq k = \frac{2(2k-1)+2}{2+2}= \frac{\lambda_1-\lambda_n}{2-\lambda_n} \leq \chi^2(G).\]
Hence,  $G$ is $2$-improper Hoffman colourable. \end{Example}

Since $\chi^{\underline{d+1}}(G) \geq \chi^d(G) \geq \frac{\lambda_1 - \lambda_n}{d-\lambda_n}$, the generalized Hoffman bound $\frac{\lambda_1 - \lambda_n}{d-\lambda_n}$ is also a lower bound for the $(d+1)$-clustered chromatic number of a graph $G$. If the Hoffman bound for the $t$-clustered chromatic number is tight for some graph $G$, a colouring with $\chi^{\underline{t}}(G)$ colours is called a \textsl{$t$-clustered Hoffman colouring}. From \Cref{Lem: construct Hoffman colourable} we obtain the following corollary.

\begin{Cor}
   If $G$ admits a Hoffman colouring, then $G \boxtimes K_{d+1}$ admits a $(d+1)$-clustered Hoffman colouring. \qedd
\end{Cor}

Since every $(d+1)$-clustered Hoffman colouring is also a $d$-improper Hoffman colouring, the graphs induced by colour classes should be $d$-regular. Since the connected components have at most $d+1$ vertices, every colour class is a disjoint union of $K_{d+1}$. Though this suggests that graphs with $(d+1)$-clustered Hoffman colourings  are strong products $H \boxtimes K_{d+1}$, of some graph $H$, this is not true and an example is given below. 

\begin{Example}
    Let $H$ be the Paley(9) graph. Take $G= L(H)$. We compute that the eigenvalues of $G$ are $6, 3^{(4)}, 0^{(4)}, (-2)^{(9)}$, where the multiplicities are given in the superscripts. Moreover, $\chi^{\underline{3}}(G) \leq 2$ see \Cref{Colouring Paley(9)}. The $(d+1)$-clustered Hoffman bound gives
\[\chi^{\underline{3}}(G) \geq \chi^2(G) \geq \frac{6+2}{2+2}=2.\]
Thus we see that $G$ has a 3-clustered Hoffman colouring and since $-1$ is not an eigenvalue of $G$ it is not isomorphic to $H \boxtimes K_3$ for any $H$.

    \begin{figure}[ht]
        \centering
        \includegraphics[width=0.3\linewidth]{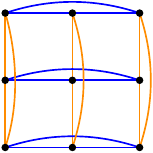}
        \caption{A $3$-clustered edge colouring of Paley(9) with 2 colours.}\label{Colouring Paley(9)}
    \end{figure}
\end{Example}

\subsection{Other spectral bounds}
The second bound that we will generalize to the $d$-improper chromatic number is the inertia bound of Cvetkovi\'{c} \cite{cvetkovic1972chromatic}. This bound uses the number of positive and negative eigenvalues of  various matrices compatible with $G$, including its adjacency matrix. In the case of $d$-improper chromatic number, we will look at the number of eigenvalues that are (strictly) greater than $d$ respectively (strictly) smaller than $-d$. The number of eigenvalues of a matrix $A$ greater than $d$ is denoted by $n_d^{+}(A)$ and $n_d^-(A)$ is used for the number of eigenvalues of $A$ smaller than $-d$. Lastly, $\alpha_d(G)$ denotes the size of the maximal induced subgraph of $G$ with maximum degree at most $d$. 

\begin{Theorem}[$d$-improper inertia bound]\label{gen inertia} Let $G$ be a graph and let $A$ be a symmetric matrix such that $|A_{u,v}| \leq 1$ for all $u$ and $v$ and $A_{u,v}=0$ if $u \not \sim v$. Then $\alpha^d(G) \leq \min\{n-n_d^+(A), n-n_d^-(A)\}$. 
\end{Theorem}
\begin{proof}
 Let $X$ be a principle submatrix of $A$ where the rows and columns correspond to  an induced subgraph of $G$ of maximum degree $d$ on $m$ vertices. Then all eigenvalues of $X$ lie in the interval $[-d,d]$ as at most $d$ of the $X_{u,v}$ per row are non-zero and the non-zero ones have absolute value at most $1$. Moreover, the eigenvalues of $A_X$ interlace those of $A$ implying
 \[
\lambda_i(A) \geq \lambda_i(X) \geq \lambda_{n-m+i}(A).
\]
 Thus, $A$ has at least $m$ eigenvalues smaller than or equal to $d$ and $A$ also has at least $m$ eigenvalues larger than or equal to $-d$. Hence, $\alpha^d \leq n-n_d^+$ and $\alpha^d \leq n-n_d^-$.
\end{proof}

\begin{Remark}
    Although this bound sounds weak, there are examples for which it is tight.  For example, take the Bowtie graph. Then the orange vertices induce a $1-$improper independent set of size $4$. The eigenvalues of this graph are $1^{(1)}$, $(-1)^{(2)}$, and  the two roots of $x^2 - x - 4$, which are  $-1.561553, 2.561553$ rounded to $6$ decimal places. Hence, the bound is tight for this graph.
\end{Remark}
\begin{figure}[ht]
    \centering
    \includegraphics[width=0.3\linewidth]{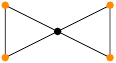}
    \caption{The Bowtie graph with a $1$-improper independent set.}
\end{figure}

\begin{Cor}\label{Cor chromatic inertia}
 Let $G$ be a graph and let $A$ be a symmetric matrix such that $|A_{u,v}| \leq 1$ for all $u$ and $v$ and $A_{u,v}=0$ if $u \not \sim v$. Then
   $\chi^d(G) \geq \left \lceil \max \left\{\frac{n}{n-n_d^+(A)},\frac{n}{ n-n_d^{-}(A)}\right\} \right\rceil$
\end{Cor}
\begin{proof}
Since $\chi^d(G) \geq \lceil\frac{n}{\alpha^d(G)} \rceil$, the result follows from \Cref{gen inertia}.
\end{proof}

\begin{Remark} The bound of \Cref{Cor chromatic inertia} is tight for the graph $2K_2 \nabla 2K_2$ (see \Cref{Join 2K2}), where $\nabla$ stands for the join. Let $A$ be the matrix in \Cref{weighted adjacency}, then $A$ satisfies all conditions of \Cref{gen inertia} for the graph $2K_2 \nabla 2K_2$. The eigenvalues of $A$, rounded to five decimal places, are
\[-2.61434, \, -2.36510, \, -1.50854, \, -1.00001, \, -0.27300, \, 1.16723, \, 2.57562,  \,4.01815.\]
In particular, $A$ has four eigenvalues smaller than $-1$ implying $\alpha^d(2K_2 \nabla 2K_2) \leq 4$ implying $\chi^1(2K_2 \nabla 2K_2)\geq \frac{8}{4}= 2$. \Cref{Join 2K2} shows that $\chi^1(2K_2 \nabla 2K_2) \leq 2$ implying that the bound is tight for the graph $2K_2 \nabla 2K$ and $d=1$.

\begin{figure}[ht]
\begin{subfigure}{0.4 \linewidth}
    \centering
     \includegraphics[width=0.55\linewidth,page= 1]{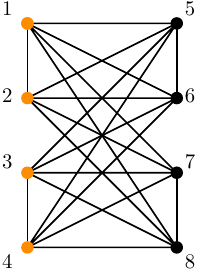}
    \caption{A $1$-improper colouring of the graph $2K_2 \nabla 2K_2$.}
    \label{Join 2K2}
\end{subfigure}\hfill
\begin{subfigure}{0.55 \linewidth}
\centering
\[\left(\begin{array}{cccccccc}
0 & 1 & 0 & 0 & 1 & 1 & -\frac{99}{100} & 1 \\
1 & 0 & 0 & 0 & 1 & 1 & 1 & 1 \\
0 & 0 & 0 & 1 & -1 & 1 & 1 & 1 \\
0 & 0 & 1 & 0 & -1 & \frac{99}{100} & 1 & 1 \\
1 & 1 & -1 & -1 & 0 & -1 & 0 & 0 \\
1 & 1 & 1 & \frac{99}{100} & -1 & 0 & 0 & 0 \\
-\frac{99}{100} & 1 & 1 & 1 & 0 & 0 & 0 & 1 \\
1 & 1 & 1 & 1 & 0 & 0 & 1 & 0
\end{array}\right)\]
\caption{A symmetric matrix $A$ such that $|A_{u,v}| \leq 1$ for all $u,v$ and $A_{u,v} =0$ if $u \not\sim v$ in the graph $2K_2 \nabla 2K_2$.}\label{weighted adjacency}
\end{subfigure}
\caption{The $1$-improper inertia bound is tight for the graph $2K_2 \nabla 2K_2$.}
\end{figure}
\end{Remark}

Also the multi-eigenvalue lower bounds of Elphick and Wocjan \cite{elphick2015unified} are generalizable to $d$-improper colourings. As these bounds are a generalization of bounds found by Nikiforov \cite{nikiforov2007chromatic} and Kolotilina \cite{kolotilina2011inequalities}, these bounds can also be adapted to hold for $d$-improper colourings. 

Before we can state our theorem, we first need to define some notation. Let $D$ be the diagonal matrix containing the degrees of the vertices of $G$, so $D_{uu} = d(u)$. Then $L=D-A$ is called the \textsl{Laplacian of $G$} and $Q=D+A$ is called the \textsl{signless Laplacian of $G$}. Suppose the eigenvalues of $L$ are denoted by $\mu_1 \geq \ldots \geq \mu_n = 0$ and the eigenvalues of $Q$ by $\theta_1 \geq \ldots \geq \theta_n \geq 0$. 
\begin{Theorem}\label{All Wocjan Elphick bounds}
 For all graphs $G$ and positive integers $d$ and $m \leq n$, the $d$-improper chromatic number of $G$ satisfies
  \begin{align}
\chi^d(G) &\geq 1+ \frac{-dm + \sum_{i=1}^m \lambda_i}{dm-\sum_{i=1}^m \lambda_{n+1-i}} \label{generalization Wocjan}\\
 \chi^d(G) &\geq 1+\frac{-dm+ \sum_{i=1}^m \lambda_i}{dm+ \sum_{i=1}^m (\mu_i -\lambda_i)}\label{generalization Nikoforov}\\ 
\chi^d(G) &\geq 1+\frac{-dm+ \sum_{i=1}^m \lambda_i}{dm+ \sum_{i=1}^m (\lambda_i+\mu_i- \theta_i)} \label{gen Kolotilina}\\ 
\chi^d(G) &\geq 1+\frac{-dm+ \sum_{i=1}^m \lambda_i}{dm+ \sum_{i=1}^m (\lambda_i + \mu_{n+1-i} - \theta_{n+1-i})}\label{generalization Kolotilina}
\end{align}
\end{Theorem}
Plugging in $m=1$ in Equation \ref{generalization Wocjan} gives the $d$-improper Hoffman bound. Similarly, $m=1$ in Equation \ref{generalization Nikoforov} corresponds to a generalization of Nikoforov's bound \cite{nikiforov2007chromatic} to the $d$-improper chromatic number and $m=1$ in Equations \ref{gen Kolotilina} and \ref{generalization Kolotilina} correspond to the $d$-improper versions of Kolotilina bounds \cite{kolotilina2011inequalities}. In Appendix \ref{More bounds}, we included the proofs of these statements, which will follow the same line of reasoning as the original proofs of Elphick and Wocjan \cite{elphick2015unified}.

\section{The \texorpdfstring{$d$-improper}{d-improper} chromatic number of \texorpdfstring{$G \boxtimes K_{d+1}$}{G box complete d+1}}\label{Section: Conjecture}

In \Cref{Cor frac makkelijk}, we have shown that the fractional $d$-improper chromatic number of $G \boxtimes K_{d+1}$ equals the fractional chromatic number of $G$. Could it be that this still holds for the corresponding integral parameters? 
We conjecture that this is the case in \Cref{conjecture d improper} (which we restate next for convenience). In this section, we will give some arguments that support our conjecture.

\conjecturedimproper*

In \Cref{colour ineq}, we have shown that $\chi(G) \geq \chi^d(G \boxtimes K_{d+1})$ holds for all graphs $G$. Therefore, to prove the conjecture one must establish the other inequality $\chi(G) \leq \chi^d(G \boxtimes K_{d+1})$. We have done this for special classes of $G$ and provide other evidence that our conjecture holds. 

First, we  look at small values for $d$. For $d=0$, a $0$-improper colouring coincides with a proper colouring and, since $G=G \boxtimes K_1$, the statement holds trivially.
For $d=1$, we observe that  the $1$-improper chromatic number coincides with the $2$-clustered chromatic number since every connected component with maximum degree 1 has size at most 2. In \Cref{Section: proof clustered}, we show that the conjecture holds if we replace the $d$-improper chromatic number with the $(d+1)$-clustered chromatic number. As a consequence, the conjecture also holds for $d=1$.

Next, we show that \eqref{equation conjecture} holds when $\chi(G) \leq 4$, using a neat argument of Buys~\cite{Pjotrcomm}. This implies that \eqref{equation conjecture} holds when $G$ is planar.

\begin{Theorem}\label{Lem: Pjotr max degree 4} For every graph $G$ with $\chi(G) \leq 4$ and every integer $d \geq 0$, we have
\[\chi(G) = \chi^d(G \boxtimes K_{d+1}).\]
\end{Theorem}
\begin{proof} 
We need only show that $\chi^d(G\boxtimes K_{d+1}) \geq \chi(G)$. We will split the argument according to the value of $\chi^d(G\boxtimes K_{d+1})$.

If $\chi^d(G\boxtimes K_{d+1}) =1$, every vertex has at most $d$ neighbours in $G\boxtimes K_{d+1}$, which implies that $G$ has no edges, and so $\chi(G)=1$.
 
Suppose we have a minimal counterexample $G$ such that $\chi^d(G\boxtimes K_{d+1}) =2$ and $\chi(G)>2$. Let $c$ be a $d$-improper colouring of $G\boxtimes K_{d+1}$ with two colours. We define $c\{v\}$ to be the set of distinct elements of  $ \{ c(v, i) \mid i = 1,\ldots, d+1\}$. 
If $|c\{v\}|=2$, then $v$ has at most one neighbour in $G$; for, if $v$ has at least two neighbours in $G$, then $| \{(u,i):u\sim v\}| \geq 2(d+1)$ and thus at least one colour appears at least $d+1$ times in this set, implying that no $(v,i)$ has that colour and $|c\{v\}|=1$. It is easy to see that a minimal counterexample has no vertices of degree one and thus $|c\{v\}|=1$ for all $v$. We can now colour every vertex $v$ with the colour in this set to obtain a proper colouring of $G$.

We will show the stronger statement that every $d$-improper colouring $c$ of $G~\boxtimes~K_{d+1}$ with three colours yields a proper colouring of $G$ where the colour of $v\in V(G)$ is  $c(v, i)$ for some $i$.

Let $H$ be a  minimal counterexample to this stronger statement. 

We again define $c\{v\}$ to be the set of distinct elements of  $\{ c(v, i) \mid i = 1,\ldots, d+~1\}$. Following the same argument as before, we see that if $v$ has at least three neighbours in $G$, then $| \{(u,i):u\sim v\}| \geq 3(d+1)$ and thus at least one colour appears at least $d+1$ times in this set, implying that no $(v,i)$ can have that colour and $|c\{v\}|\leq 2$. Thus if $|c\{v\}| = 3$, then $v$ has at most two neighbours in $H$ and thus $H$ is not a minimal counterexample. 

If there exists $v$ with $|c\{v\}|=1$, then $c\{v\}$ is disjoint from $c\{u\}$ for any neighbour $u$ of $v$. Therefore, $H-v$ would also be a counterexample to this statement. Therefore, $|c\{v\}|=2$ for every $v$. Suppose that $u \sim v$ and $c\{v\}=c\{u\}$. Let $w\notin \{u,v\}$ be adjacent to at least one of $u$ or $v$, then $c\{w\}$ can only contain the colour that is not in $c\{v\}=c\{u\}$. This is a contradiction, since no vertex $w \in V(H)$ has $|c\{w\}| = 1$. Thus, $u$ and $v$ for a connected component of $H$, isomorphic to $K_2$, and $H-u-v$ is also a counterexample. Now it must be the case that $|c\{v\}|=2$ for every $v$ and $c\{v\} \neq c\{u\}$ if $u \sim v$. Colouring the vertices with the set $\{1,2\}$ with colour 1, with set $\{1,3\}$ with colour 3 and $\{2,3\}$ with colour 2, will now result in a proper colouring of $H$.

Thus if $\chi(G) > \chi^d(G\boxtimes K_{d+1})$, we must have $ \chi^d(G\boxtimes K_{d+1}) \geq 4$ and $\chi(G) \geq 5$.
\end{proof}

Next, we will compare the expressions for our $d$-improper spectral bounds of $\chi^d(G \boxtimes K_{d+1})$ that we found in \Cref{Section: spectral bounds} to spectral bounds for $\chi(G)$. We will see that these bounds often coincide. To compare these bounds, we will need to express the eigenvalues of $A(G \boxtimes K_n), L(G~\boxtimes~K_n)$ and $Q(G \boxtimes K_n)$ in terms of those of $A(G), L(G)$, and $Q(G)$ respectively, which we give in \Cref{tab:eigenvalue-conversion}.

\begin{table}[ht]\begin{tabular}{c|c|c |c}
   & Adjacency matrix & Laplacian matrix & Signless Laplacian matrix \\
   \hline
  $G$ & $\lambda_i$ &$\mu_i$ & $\theta_i$\\
  \hline
  \multirow{2}{*}{$G \boxtimes K_n$} & $n\lambda_i+(n-1)$ & $n\mu_i$ & $n\theta_i+2(n-1)$ \\
  & $-1$ & $nd(u_i)+n$ & $(n-2)+nd(u_i)$
\end{tabular}
\caption{The eigenvalues of the different matrices, where $d(u_i)$ is the degree of $u_i$ and $i=1,\ldots, |V(G)|$. This table does not include the multiplicities of these eigenvalues, but every eigenvalue appears in this table \cite{barik2015laplacian}.
\label{tab:eigenvalue-conversion} }
\end{table}

\begin{Remark} Let $G$ be non-empty. Since $\lambda_1 > -1 \geq \lambda_n$, the largest eigenvalue of $G\boxtimes K_d$ is $(d+1)\lambda_1+d$ and the smallest eigenvalue of $G\boxtimes K_d$ is $(d+1)\lambda_n+d$. Now, applying the Hoffman bound for $d$-improper colourings to $G \boxtimes K_{d+1}$ gives us \[
 \chi^d(G \boxtimes K_{d+1}) \geq \frac{(d+1)\lambda_1+d-((d+1)\lambda_n+d)}{d-((d+1)\lambda_n+d)} = \frac{\lambda_1-\lambda_n}{-\lambda_n}= 1-\frac{\lambda_1}{\lambda_n}.\] We obtain that the (generalized) Hoffman bound for $\chi^d(G \boxtimes K_{d+1})$ equals the Hoffman bound for $\chi(G)$.
\end{Remark}

As a consequence, \eqref{equation conjecture} holds for graphs that are Hoffman colourable since
\[1-\frac{\lambda_1}{\lambda_n} = \chi(G) \geq \chi^d(G \boxtimes K_{d+1}) \geq 1-\frac{\lambda_1}{\lambda_n}\]
whence all inequalities must hold with  equality. Applying this trick to equation \ref{first inequality d-improper} gives that all graphs satisfying $\chi(G) = \omega(G)$ also satisfy
\[\omega(G) = \chi(G) \geq \chi^d(G \boxtimes K_{d+1}) \geq \frac{\omega(G \boxtimes K_{d+1})}{d+1}= \frac{\omega(G) \cdot (d+1)}{d+1}=\omega(G).\]
Thus, \eqref{equation conjecture} holds for all $G$ such that $\chi(G) = \omega(G)$.  In particular, our conjecture holds for all perfect graphs.

\begin{Remark}
Since every graph $G$ with at least one edge satisfies $\mu_1(G) \geq \Delta(G)+1$ \cite{grone1994laplacian}, we obtain that $\mu_1(G \boxtimes K_{d+1})=(d+1)\mu_1$. Hence, the $d$-improper version of Nikoforov bound (Equation \ref{generalization Nikoforov} with $m=1$) gives that
\[\chi^d(G \boxtimes K_{d+1}) = 1 + \frac{-d+(d+1)\lambda_1+d}{d+(d+1)\mu_1-((d+1)\lambda_1+d)} \geq 1+\frac{\lambda_1}{\mu_1-\lambda_1},\] which is exactly Nikiforov's bound for $\chi(G)$.
\end{Remark}

\begin{Remark}Furthermore, since $\theta_1 \geq \mu_1 \geq \Delta(G)+1$, we have that $\theta_1(G \boxtimes K_{d+1})=(d+1)\theta_1+2d$. Thus \eqref{gen Kolotilina} with $m=1$, gives that
\[
\begin{split}
    \chi^d(G \boxtimes K_{d+1}) &\geq 1+ \frac{-d+(d+1)\lambda_1+d}{d+(d+1)\lambda_1+d+(d+1)\mu_1-((d+1)\theta_1+2d)}\\
    &=1+\frac{\lambda_1}{\lambda_1+\mu_1-\theta_1},
\end{split}\]
  which is exactly one of Kolotilina's bounds for $\chi(G)$. 
\end{Remark}

\begin{Remark}\label{Last positive remark} Lastly, in  \eqref{generalization Kolotilina} with $m=1$, the smallest eigenvalue of the Laplacian matrix is $0$. Hence, we only obtain the same expression when the smallest eigenvalue of the signless Laplacian of $G \boxtimes K_{d+1}$ is $n\theta_n+2(n-1)$. This is the case if $$n\theta_n+2(n-1) \leq (n-2)\delta(G)$$  or, equivalently, if $\theta_n \leq \delta(G)-1$, where $\delta(G)$ is the minimal degree of $G$. While it is true that $\theta_n \leq \delta(G)$, there are graphs (for example, $K_5$ with a pendant vertex) such that $\theta_n \leq \delta(G)-1$ does not hold.

If equality holds in Kolotilina's bound, then $\lambda_1(G) = \Delta(G)$, which implies that $G$ is regular.  For regular graphs, the inequality $\theta_n \leq \delta(G)-1$ holds, and thus \eqref{equation conjecture} holds for graphs for which Kolotilina's bound is tight.
\end{Remark}

\subsection{The \texorpdfstring{$t$-clustered chromatic numbers}{c-clustered chromatic numbers} of \texorpdfstring{$G \boxtimes K_t$ }{G box complete graph }}\label{Section: proof clustered} 

In this section, we prove a relaxation of \Cref{conjecture d improper} by proving it for the $(d+1)$-clustered chromatic number instead of the $d$-improper chromatic number.

Let us first recall the notation from the proof of \Cref{Lem: Pjotr max degree 4}. Given an arbitrary colouring $c$ of $G \boxtimes K_{t}$, for every $v \in V(G)$ we define $c[v]$ to be the multiset $\{c((v, i)) \mid i=1,\ldots,t\}$ and $c\{v\}$ to be its underlying set.

The following observation is useful for us.
To determine whether a colouring is an $\ell$-clustered colouring, we only need to consider the multisets $c[v]$, since the closed neighbourhood of $(v,j)$ is the same for all $j=1,\dots,t$.

We will specifically prove that, starting with an $\ell t$-clustered colouring of $G \boxtimes K_t$, we are able to obtain an $\ell$-clustered colouring of $G$ by picking colours only from $c\{v\}$.

Our strategy is to prune the original $\ell t$-clustered colouring to  a `nice' colouring so that we can pick our colours using a naive algorithm.
In order to characterize nice colourings, we will introduce an auxiliary graph $\mathcal{I}_{G,c}$ based on a colouring $c$ of $G \boxtimes K_{t}$. The vertices of $\mathcal{I}_{G,c}$ are the vertices of $G$ together with some new vertices $v_C$, where $C$ is a monochromatic connected component of $G \boxtimes K_t$ in the colouring $c$. We let $v\in V(G)$ be adjacent to $v_C$ if there is some $i$ for which $(v,i) \in C$.
In this case, we say that $C$ \textsl{covers} $v$.

We now show how to transform a clustered colouring of $G \boxtimes K_{t}$ into a `nice' colouring.

\begin{Lemma}\label{no hypercycles} 
 If $c$ is an $\ell$-clustered colouring of $G \boxtimes K_{t}$, then there exists some $\ell$-clustered colouring $c_1$ of $G \boxtimes K_{t}$ such that:
 \begin{enumerate}[1)]
  \item $c_1\{v\} \subseteq c\{v\}$ for all $v$, and
  \item $\mathcal{I}_{G,c_1}$ is acyclic.
 \end{enumerate}
 In particular, the number of colours used in $c_1$ is at most the number of colours used in $c$.
\end{Lemma}
\begin{proof} We will prove that we can eliminate all cycles iteratively. Let $v_1v_{C_1}v_2\ldots v_{C_k}v_{k+1}$, where $v_{k+1}=v_1$, be a smallest cycle in $\mathcal{I}_{G,c_1}$. Let us suppose that the component $C_i$ is monochromatic in the colour $a_i$ for all $i=2,\dots,k+1$.  By the definition of $\mathcal{I}_{G,c}$, we have $a_i, {a_{i-1}} \in c\{v_i\}$ for all $i$. Suppose that $s= \min_i \, (\text{multiplicity of } a_{i} \text{ in } c[v_i] ) $ is attained by $i=j$.

Then we cyclically redistribute the colours according to the following procedure. 
For $i=2\dots,k+1$, we take $s$ of the vertices $(v_i,x)$ of colour $a_{i}$ and recolour them with colour $a_{i-1}$.

Since $a_{i-1} \in c\{v_i\}$ for all $i$, we have $c'\{v_i\} \subseteq c\{v_i\}$ for all $i$.
One important reason for cyclically redistributing the colours in this way is that, by the choice of $j$, $c'[v_j]$ does not have any $a_j$. Hence,  $\mathcal{I}_{G,c'}$ does not contain the edge $v_jv_{C_j}$, nor does it contain the cycle $v_1v_{C_1}v_2\ldots v_{C_k}v_{k+1}$.

Since $c'\{v_i\} \subseteq c\{v_i\}$ for all $i$, we have that if a monochromatic component $C'$ of colour $a$ in the colouring $c'$ covers both $v_i$ and $v_{i'}$, then there is also a monochromatic component $C$ of colour $a$ in the colouring $c$ that covers both $v_i$ and $v_{i'}$.

It follows that, if $v_i$ and $v_{i'}$ are both adjacent to $v_{C'}$ in $\mathcal{I}_{G,c'}$ where $C'$ has colour $a$ in $c'$, then $v_i$ and $v_{i'}$ are both adjacent to some vertex $v_{C}$ in $\mathcal{I}_{G,c}$ where $C$ has colour $a$ in $c$. In both $\mathcal{I}_{G, c}$ and $\mathcal{I}_{G,c'}$ every vertex $v_i$ is adjacent to at most one vertex $v_{C'}$ corresponding to a colour $a$. 

Every closed walk in $\mathcal{I}_{G,c'}$ corresponds to a closed walk in $\mathcal{I}_{G,c}$ by replacing each vertex of the form $v_{C'}$ by a corresponding vertex of the form $v_C$ as implied by the statements above. Since the neighbours $v_{C_{i-1}'}$ and $v_{C'_i}$ of every $v_{i'}$ correspond to different colours in $\mathcal{I}_{G,c'}$, the vertices $v_{C_{i-1}}$ and $v_{C_i}$ are different in $\mathcal{I}_{G,c}$ for every $i$. Hence, there are no additional cycles in $\mathcal{I}_{G,c'}$ as compared to $\mathcal{I}_{G,c}$.

What remains to prove is that $c'$ is an $\ell$-clustered colouring of $G \xbox K_t$. For every $i$ all vertices $(v_i,x)$ and $(v_{i+1}, x')$ with colour $a_i$ belong to the monochromatic connected component $C_i$. Suppose that the monochromatic component $C_i$ of colour $a_i$ in the colouring $c$ covers the vertex set $W \subseteq V(G)$. By the cyclic redistribution, the total multiplicity of $a_i$ within the vertices of $W$ did not change. Recall that there exists a monochromatic component $C$ of colour $a$ in the colouring $c$ that covers both $v_i$ and $v_{i'}$ for each monochromatic connected component $C'$ of colour $a$ in the colouring $c'$ that covers both $v_i$ and $v_{i'}$. Thus, no vertex outside $W$ lies in a monochromatic component of colour $a_i$ in the colouring $c'$ with a vertex of $W$. Hence, every monochromatic connected component has size at most $\ell$.

By repeating this procedure,  we eventually obtain an $\ell$-clustered colouring $c_1$ of $G \boxtimes K_{t}$ such that $\mathcal{I}_{G,c_1}$ is acyclic.
\end{proof}

We will now prove that if for some $\ell$-clustered colouring $c$, the graph $\mathcal{I}_{G,c}$ is acyclic, then $G \boxtimes K_t$ under colouring $c$ contains a monochromatic component that covers few vertices of $G$. This will be a subroutine in our naive colouring algorithm.

\begin{Lemma}\label{small conn component} If $c$ is an $\ell t$-clustered colouring of $G \boxtimes K_{t}$ such that $\mathcal{I}_{G,c}$ acyclic, then there exists a monochromatic component $C$ that covers at most $\ell$ vertices of $G$.
\end{Lemma}
\begin{proof}
Suppose to the contrary that every monochromatic component of $G \boxtimes K_{t}$ covers at least $\ell+1$ vertices of $G$. As every vertex in $G$ corresponds to $t$ vertices of $G \boxtimes K_t$ and every monochromatic component contains at most $\ell t$ vertices in $G \boxtimes K_t$, every monochromatic component $C_1$ covers at least two vertices $v_1$ and $v_2$ of $G$ that both are covered by another monochromatic component, say, $C_0$ (respectively $C_2$). 

Notice that the vertex $v_{C_1}$ is both adjacent to $v_1$ and $v_2$ in $\mathcal{I}_{G,c}$ and $v_2$ is also adjacent to $v_{C_2}$. By the same argument, the connected component $C_2$ covers at least two vertices that lie in other components. In particular, it covers a different vertex $v_3 \neq v_2$ that lies in some other monochromatic component $C_3$. As there exists such a vertex for every new monochromatic component, we can construct an infinite walk $v_1 v_{C_1} v_2 v_{C_2} v_3 v_{C_3} \ldots$ in $\mathcal{I}_{G,c}$ where $v_i \neq v_{i+1}$ and $C_{i} \neq C_{i+1}$ for all $i$.

Since $\mathcal{I}_{G,c}$ is acyclic all vertices in this walk must be distinct. Since every other vertex of the sequence lies in $G$, we must have that $G$ has infinitely many vertices, a contradiction.
\end{proof}

\begin{Theorem}\label{clustered ell versie}
 For every graph $G$ and positive integers $t$ and $\ell$, we have
 \[\chi^{\underline{\ell}}(G) = \chi^{\underline{\ell t}}(G \boxtimes K_{t}).\] 
\end{Theorem} 
\begin{proof}
We have $\chi^{\underline{\ell}}(G) \geq \chi^{\underline{\ell t}}(G \boxtimes K_t)$ as we can copy every colour of $G$ exactly $t$ times. Therefore, the rest of the proof is dedicated to showing $\chi^{\underline{t}}(G \boxtimes K_t) \geq \chi(G)$.

\Cref{no hypercycles} gives us that $G \boxtimes K_t$ has a colouring $c$ with $\chi^{\underline{\ell t}}(G \boxtimes K_t)$ colours such that $\mathcal{I}_{G,c}$ is acyclic. By \Cref{small conn component}, there exists at least one monochromatic component that covers at most $\ell$ vertices.\\

We will now repeatedly apply the following steps. 
\begin{enumerate}[1)]
 \item Pick an arbitrary monochromatic component of size at most $\ell$ corresponding to some colour $a$. 
 \item Colour all vertices in this monochromatic component with the colour $a$ and delete them afterwards.
\end{enumerate}
If we delete these vertices we obtain a new graph $G'$ and the graph $\mathcal{I}_{G',c}$ is acyclic (as it is a subgraph of $\mathcal{I}_{G,c}$). Moreover, $c$ restricted to  $G'\boxtimes K_t$ is still and $\ell t$-clustered colouring. Hence, by \Cref{small conn component}, there is a monochromatic connected component that covers at most $\ell$ vertices. Hence, we can repeat these steps until we have coloured all vertices.

By definition of our algorithm, every uncoloured vertex has some preliminary colours from an $\ell t$-improper colouring of $G' \boxtimes K_t$. Therefore, we will colour all vertices. Moreover, if two neighbours have the same colour $a$, they must have been in the same connected component of size at most $\ell$. Therefore, we have an $\ell$-clustered colouring of $G$ with $\chi^{\underline{\ell t}}(G \boxtimes K_t)$ colours implying $\chi^{\underline{\ell}}(G) \leq \chi^{\underline{\ell t}}(G \boxtimes K_{t})$.
\end{proof}

As an immediate consequence, we derive \Cref{conjecture d improper} where we have replaced that $t-1$-improper chromatic number with the clustered chromatic number:
\begin{Cor}\label{cluster strong}
 For every graph $G$ and positive integer $t$, we have
 \[\chi(G) = \chi^{\underline{t}}(G \boxtimes K_{t}).\]
\end{Cor}
\begin{proof}
 We take $\ell=1$ in \Cref{clustered ell versie}. Every graph $G$ satisfies $\chi^1(G)=\chi(G)$, as $1$-clustered and proper are equivalent. Therefore $\chi(G)=\chi^{\underline{t}}(G \boxtimes K_{t})$.
\end{proof}

As a corollary to \Cref{cluster strong}, we obtain the statement for $b$-fold colourings.

\begin{Cor} For every graph $G$ and all positive integers $t$ and $b$, we have
 \[\chi_b(G) = \chi^{\underline{t}}_b(G\boxtimes K_t).\]
\end{Cor}
\begin{proof}
Since $\chi_b(G) = \chi(G \boxtimes K_b)$ for every $b$, \Cref{b fold cluster erin} gives that \[\chi^{\underline{t}}_b(G\boxtimes K_t) = \chi^{\underline{t}}(G\boxtimes K_t \boxtimes K_b) = \chi^{\underline{t}}(G\boxtimes K_b \boxtimes K_t)\] by associativity of the strong product. Hence, the result follows from applying \Cref{cluster strong} to the graph $G \boxtimes K_b$. 
\end{proof}

Lastly, we can prove a generalized version of \Cref{clustered ell versie} using lexicographical products denoted by $G[H]$. The \textsl{lexicographical product $G[H]$} has $G \times H$ as vertex set and $(v, i) \sim (w,j)$ if either $v \sim w$ or $v=w$ and $i \sim j$. In particular, $E(G \boxtimes H) \subseteq E(G[H])$ and there is an equality if and only if $H$ is a complete graph.

\begin{Cor}\label{lexi prod} Let $H$ be a graph on $t$ vertices. Then for every positive integer $\ell$ and graph $G$ we have
 \[\chi^{\underline{\ell}}(G) = \chi^{\underline{\ell t}}(G[H]).\]  
\end{Cor}
\begin{proof}
Since $\ell t \geq t$ for all positive integers $\ell$, we have by Lemma \ref{lemma lexi prod} and Theorem \ref{clustered ell versie} that 
 \[\chi^{\underline{\ell t}}(G[H]) = \chi^{\underline{\ell t}}(G \boxtimes K_t) = \chi^{\underline{\ell}}(G). \qedhere\]
\end{proof}

\appendix

\section{Generalizing bounds of Wocjan and Elphick}\label{More bounds}
The spectral bounds from \textsl{`Unified spectral bounds on the chromatic number'} \cite{elphick2015unified} can be easily generalized to hold for the $d$-improper chromatic number. In this section, we will give a proof of \Cref{All Wocjan Elphick bounds} that follows roughly the same reasoning as the proof in their paper. 

We use $\lambda^{\downarrow}_i(A)$ to denote the $i$-th largest eigenvalue of a Hermitian matrix $A$. Moreover, we will use the following result about Hermitian matrices that is often used in majorization theory \cite{bhatia2013matrix}.

 Let $A_1, \ldots, A_l$ be arbitrary $n \times n$ Hermitian matrices, then for all $m \geq 0$ we have that
\[\sum_{i=1}^m \sum_{j=1}^l \lambda_i(A_j) \geq \sum_{i=1}^m \lambda_i\left(\sum_{j=1}^l A_j\right). \]

\begin{Lemma}\label{lemma wocjan elphick}
Let $G$ be a graph such that $\chi^d(G) = c$ and let $A$ be the adjacency matrix of $G$. Then for any real diagonal matrix $X$ 
\[\sum_{i=1}^m \lambda_i^{\downarrow}(X-A) \geq \sum_{i=1}^m \lambda_i^{\downarrow}\left(X+\frac{A}{c-1}\right) -\frac{cdm}{c-1} \]
for all $m=1,\ldots, n$.
\end{Lemma}
\begin{proof}
Let $\phi: G \rightarrow [c]$ be a $d$-improper colouring and let $\zeta$ be a primitive $c$-th root of unity. For $s=1,\ldots c$ we define the $n \times n$-matrices $U_s = \diag(\zeta^{\phi(1) \cdot s}, \ldots, \zeta^{\phi(n) \cdot s})$. Moreover, we define $B = \sum_{s=1}^c U_s A U_s^\dag$.

Multiplication of $A$ by the diagonal unitary matrix $U_s$ from the left corresponds to the multiplication of the $k$th row of $A$ by the $k$th diagonal entry of $U_s$. Moreover, the multiplication of $A$ by the diagonal unitary matrix $U_s^\dag$ equals the multiplication of the $l$th column of $A$ by the $l$th diagonal entry of $U_s$. Hence, 
\[B_{kl} = \sum_{s=1}^c A_{kl} \zeta^{(\phi(k)-\phi(l)) s} \]
for $k,l=1, \ldots, n$.
If $\phi(k)\neq \phi(l)$ we have $\sum_{s=1}^c \zeta^{(\phi(k)-\phi(l))s}=0$ and thus $B_{kl}=0$. If $\phi(k) = \phi(l)$ then $\sum_{s=1}^c \zeta^{(\phi(k)-\phi(l))s}=\sum_{s=1}^c 1=c$. In conclusion, 
\[B_{kl} = \begin{cases} c & \text{if } \phi(k)=\phi(l) \text{ and } k \sim l\\
 0 & \text{else}\end{cases}.\]
As $\phi$ is a $d$-improper colouring, $B$ has at most $d$-non zero entries per row and column and those are all equal to $c$. Hence, $\lambda_1(B) \leq cd$. Lastly, as $U_{c}=I$, we obtain that $\sum_{s=1}^{c-1} -U_sAU_s^\dag=A-B$.

Let $X$ be any real diagonal matrix. As $U_s$ and $U_s^{\dag}$ are also diagonal matrices, we see that $X, U_s$ and $U_s^{\dag}$ commute. Therefore,
\[\sum_{s=1}^{c-1} U_s(X-A)U_s^{\dag} = (c-1)X+A-B.\]

 Since $X, A, U_s$ and $U_s^{\dag}$ are all Hermitian matrices, we can apply the majorization principle to obtain
 \begin{align*}(c-1)\sum_{i=1}^m \lambda_i^{\downarrow}(X-A) &= \sum_{i=1}^m \sum_{s=1}^{c-1} \lambda_i^{\downarrow}(U_s(X-A)U_s^{\dag})\\ 
 &\geq \sum_{i=1}^m \lambda_i^{\downarrow}\left((c-1)X+(A-B)\right)\\
 &\geq \sum_{i=1}^m \lambda_i^{\downarrow}\left((c-1)X+A)\right)-\sum_{i=1}^m \lambda_i(B)\\
 &\geq (c-1)\sum_{i=1}^m \lambda_i^{\downarrow}\left(X+\frac{1}{c-1}A\right)-cdm. \end{align*}
Dividing by $c-1$ yields the result.
\end{proof}

\begin{Cor}\label{Cor wocjan important}
 Let $A$ be the adjacency matrix of a graph $G$. Then for any real diagonal matrix $X$ and integer $d\geq 0$, we have
 \[ \chi^d(G) \geq 1+\frac{-dm+\sum_{i=1}^m \lambda_i^{\downarrow}(A)}{dm+ \sum_{i=1}^m \lambda_i^{\downarrow}(X-A)+ \lambda_i^{\downarrow}(A)- \lambda_i^{\downarrow}(X+A)}\]
 for all $m=1, \ldots, n$.
\end{Cor}
\begin{proof}
 First, let $c= \chi^d(G)$. We now write $X+\frac{1}{c-1}A= X+A - \frac{c-2}{c-1}A$. By combining this with the previous \Cref{lemma wocjan elphick} and majorization principle for Hermitian matrices, we obtain
 \begin{align*}\sum_{i=1}^m \lambda_i^{\downarrow}(X-A) &\geq \sum_{i=1}^m \lambda_i^{\downarrow}\left(X+\frac{A}{c-1}\right)-\frac{cdm}{c-1}\\
 &= \sum_{i=1}^m \lambda_i^{\downarrow}\left(X+A-\frac{c-2}{c-1}A)\right)-\frac{cdm}{c-1}\\
 &\geq \sum_{i=1}^m \lambda_i^{\downarrow}(X+A) - \frac{c-2}{c-1}\sum_{i=1}^m \lambda_i(A)-\frac{cdm}{c-1}.
 \end{align*}
 After multiplying everything by $c-1$, we can eliminate $c$ to obtain
 \[c-1 \geq \frac{-dm+ \sum_{i=1}^m \lambda_i^{\downarrow}(A)}{dm+ \sum_{i=1}^m \lambda_i^{\downarrow}(X-A)+ \lambda_i^{\downarrow}(A)- \lambda_i^{\downarrow}(X+A)}. \qedhere\]
\end{proof}

With this Lemma, we can prove Equation \ref{generalization Wocjan} which is the $d$-improper version of the bound proven in \cite{wocjan2013new}:
\begin{Theorem}\label{gen wocjan}
 For every graph $G$ and positive integers $d$ and $m \leq n$ we have \[\chi^d(G) \geq 1+ \frac{-dm + \sum_{i=1}^m \lambda_i}{dm-\sum_{i=1}^m \lambda_{n+1-i}}.\]
\end{Theorem}
\begin{proof}
Plug in $X=0$ into \Cref{Cor wocjan important}. Then we obtain
 \[c \geq 1+ \frac{\sum_{i=1}^m \lambda_i(A) -dm}{\sum_{i=1}^m\lambda_i(-A)+dm} = 1+ \frac{-dm + \sum_{i=1}^m \lambda_i}{dm-\sum_{i=1}^m \lambda_{n+1-i}}. \qedhere \]
\end{proof}
This bound becomes the $d$-improper Hoffman bound when we plug in $m=1$. Next, we will give an example that shows that for every $m$ there exist graphs for which this bound is (slightly) better than the $d$-improper Hoffman bound.

\begin{Remark}
 Let $G$ be a bipartite graph. Let $G \boxtimes K_{n}$. For $k$ satisfying $\lambda_k \geq 1$, the largest $k$ eigenvalues of $G \boxtimes K_n$ are equal to $n\lambda_i+(n-1)$ while the smallest $k$ eigenvalues are equal to $-n(\lambda_i)+n-1$. Hence, for $m=k$ the Wocjan fraction equals $\frac{2n(\lambda_1+\ldots \lambda_k)}{k(d-n+1)+n(\lambda_1+\ldots+ \lambda_k)}$. 

  If $\lambda_k \geq 1$, we have $(k-1)\lambda_k < \lambda_1+\ldots +\lambda_{k-1}$ as the largest eigenvalue has multiplicity 1. Hence, if $d< n-1$ we find
 \[2n\lambda_k((k-1)(d-n+1) > 2n(\lambda_1+\ldots + \lambda_k)(d-n+1)\]
implying
 \[2n\lambda_k((k-1)(d-n+1)+n(\lambda_1+\ldots \lambda_{k-1}) > 2n(d-n+1+n\lambda_k)(\lambda_1+\ldots+ \lambda_{k-1})\]
 and thus
 \[\frac{2n(\lambda_1+\ldots \lambda_k)}{k(d-n+1)+n(\lambda_1+\ldots+ \lambda_k)} > \frac{2n(\lambda_1+\ldots \lambda_{k-1})}{(k-1)(d-n+1)+n(\lambda_1+\ldots+ \lambda_{k-1})}.\]

Hence, for these graphs, the generalized Wocjan bound is strictly better than the Hoffman bound.
 \end{Remark}

Recall that if $D$ is the diagonal matrix with the degrees of the vertices of $G$ as its entries, then $L=D-A$ is the Laplacian of $G$ and $Q=D+A$ is the signless Laplacian of $G$. The eigenvalues of $L$ are denoted by $\mu_1 \geq \ldots \geq \mu_n = 0$ and the eigenvalues of $Q$ by $\theta_1 \geq \ldots \geq \theta_n \geq 0$. By plugging in $X=D$, we obtain Equation \ref{gen Kolotilina} which is a generalization of bound $(7)$ in \cite{elphick2015unified}.
\begin{Cor}
 For every graph $G$ and positive integers $d$ and $m \leq n$ we have 
 \[ \chi^d(G) \geq 1+\frac{-dm+ \sum_{i=1}^m \lambda_i}{dm+ \sum_{i=1}^m (\lambda_i+\mu_i- \theta_i)}\]
 for all $m=1, \ldots, n$.
\end{Cor}

Using that $\theta_i \geq \lambda_i$ for all $i$, we get Equation \ref{generalization Nikoforov}. Lastly, when we plug in $X=-D$, we obtain Equation \ref{generalization Kolotilina} which is a generalization of bound $(8)$ in \cite{elphick2015unified}.
\begin{Cor}
 For every graph $G$ and positive integers $d$ and $m \leq n$ we have 
 \[ \chi^d(G) \geq 1+\frac{-dm+ \sum_{i=1}^m \lambda_i}{dm+ \sum_{i=1}^m (\lambda_i + \mu_{n+1-i} - \theta_{n+1-i})}\]
 for all $m=1, \ldots, n$.
\end{Cor}

\end{document}